%% file: houghton-conjugacy_resub2.tex
\documentclass[11pt,letterpaper]{amsart}
\usepackage{amsfonts, amsmath, amssymb, amscd, amsthm, graphicx, mathrsfs, wasysym, setspace, mdwlist, color}
\usepackage{hyperref}
\newtheorem{thm}{Theorem}[section]
\newtheorem{cor}[thm]{Corollary}
\newtheorem{lem}[thm]{Lemma}
\newtheorem{prop}[thm]{Proposition}

\theoremstyle{definition}
\newtheorem{defn}[thm]{Definition}

\theoremstyle{remark}
\newtheorem{rem}[thm]{Remark}

\def \Z{{\mathbb Z}}

\def\N{{\mathbb N}}           
\newcommand{\eps}{\varepsilon}
\newcommand{\cO}{\mathcal{O}}

\DeclareMathOperator{\Sym}{Sym}
\DeclareMathOperator{\FSym}{FSym}
\DeclareMathOperator{\supp}{supp}

\begin{document}

\author{Y. Antol\'in and J. Burillo and A. Martino}

\title{Conjugacy in Houghton's Groups}
\date{\today}
\maketitle
\begin{abstract}
Let $n\in \N$. Houghton's group $H_n$ is the group of permutations of $\{1,\dots, n\}\times \N$,
that eventually act as a translation in each copy of $\N$. We prove the solvability of
the conjugacy  problem and conjugator search problem for $H_n$, $n\geq 2$.
\medskip

{\footnotesize
\noindent \emph{2010 Mathematics Subject Classification.} Primary: 20F10, 20B99. 

\noindent \emph{Key words.} Houghton group, Permutation groups, Conjugacy problem.}

\end{abstract}

\section{Introduction}

As Bourbaki intended, we let $\N$ denote the set of finite cardinals, $\{0,1,2,\ldots \}$.

Fix a positive integer $n$. We use $X_n$ to denote the set $\{ 1, \ldots, n \} \times \N$.
Elements of $X_n$ will be usually denoted as pairs $(i, m)$, $i\in \{1, \ldots, n\}, m\in \N$.

Throughout, $\Sym=\Sym(X_n)$ will denote the group of all permutations of the set $X_n$, and
$\FSym=\FSym(X_n)$ the subgroup of permutations with finite support.

In \cite{Houghton} Houghton introduced the group $H_n$ of permutations of $X_n$ that are ``eventually translations"
in each copy of $\mathbb{N}$. More precisely, an element $g$ of $\Sym$ lies in $H_n$ if and only if there exists
a positive integer $z$ and integers $t_1(g), \ldots, t_n(g)$ such that:

\begin{equation}
\label{eq:translation}
(i,m)g = (i, m+ t_i(g) ), \, \forall m \geq z.
\end{equation}

The groups $H_n$, $n=3,4,\ldots$ are known as {\it Houghton's groups} and they  are finitely generated.
This is not at all obvious from our definition, but it is clear form the original definition of Houghton, who defined $H_n$
as the  subgroups of $\Sym$ generated by $g_2,\dots, g_{n}$ where $g_i$ acts on the first copy of $\N$ by pushing elements one step to infinity and acts
on the elements of the $i$th copy of $\N$ by pushing elements one step towards the origin. See Figure \ref{fig:fig1}.
In cycle notation we have that

\begin{equation}\label{eq:generator}
g_i=\ldots (i,2),(i,1) ,(i,0),(1,0),(1,1),(1,2),\ldots
\end{equation}

\begin{figure}[ht]
\begin{center}
\scalebox{.5}{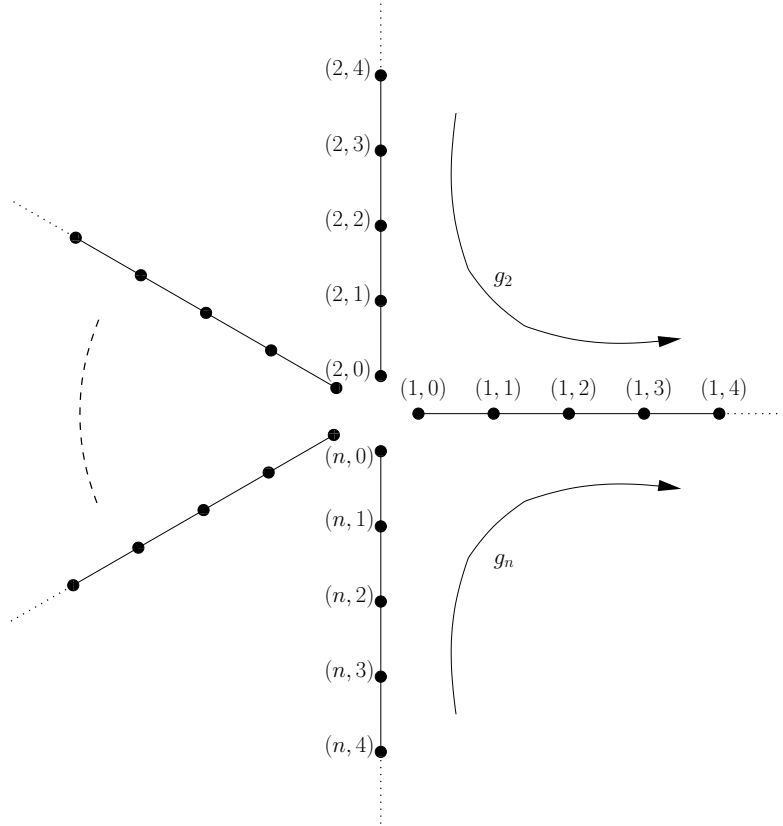}
\caption{The set $X_n$ and the action of the generators $g_2$ and $g_n$.}
\label{fig:fig1}
\end{center}
\end{figure}

The integers $t_i(g)$ of \eqref{eq:translation} define a map $H_n \to \Z^n$, assigning to each element its
``eventual translation lengths''. It is clear from the description of the generators that the image of $t$ is  the subgroup
$\{ (x_1, \ldots, x_n) \in \Z^n \ : \ \sum_{i=1}^n x_i =0 \}$ of $\Z^n$.

This subgroup is isomorphic to $\Z^{n-1}$ and it was Wiegold in \cite{Wiegold} who proved that
 we have a short exact sequence

$$
1 \longrightarrow \FSym \longrightarrow H_n \longrightarrow \Z^{n-1} \longrightarrow 1.$$

\medskip

One can extend the definition given by Houghton to the cases $n=1,2$, by setting $H_1=\FSym$ and $H_2$ being described as the group
of permutations of $X_2$ that are eventually translations. However, we notice that one needs a slightly
different generating set for $H_2$. It is not hard to see that $H_2$ coincides with the subgroup of
$\Sym(X_2)$ generated by $g_2$ (as in \eqref{eq:generator}) and the transposition $((1,0),(2,0))$.

In \cite{Johnson}, Johnson found a presentation for $H_3$.
One of the main interesting properties of Houghton's groups was discovered by Brown in \cite{Brown87}.
Brown proved that, for $n\geq 1$, $H_n$ is $FP_{n-1}$ but not $FP_{n}$.
In particular $H_1$ is not finitely generated, $H_2$ is finitely generated but not finitely presented,
and for $n>2$ $H_n$ is finitely presented.

Presentations of Houghton's groups and an alternative proof of Brown's results can be found in \cite{SRLee}.
In \cite{st.john-green}, St.John-Green studied Bredon finiteness conditions for Houghton's groups.
Finally, we remark that R\"over \cite{Rover}, showed that for all $n\geq 1, r\geq 2, m\geq 1$,
$H_n$ embeds in Higman's groups $G_{r,m}$ (defined in \cite{Higman}), and in particular
all Houghton's groups are subgroups of Thompson's group $V$.

\begin{defn}
Let  $\mathcal{P}=\langle S\,|\,R \rangle$ be a group presentation, and $G$ the corresponding group.

The presentation has {\it solvable conjugacy problem} if there exists an algorithm that
given two words over $S$, the algorithm decides whether or not these two words correspond to conjugate elements
in the group.

\end{defn}
The solvability of the conjugacy problem is in fact an algebraic property of the group, i.e. it is independent of
the presentation of the group and clearly it implies the solvability of the word problem. See \cite{Miller}.

In this paper we investigate the conjugacy problem for Houghton's groups $H_n$, $n\geq 2$.

Our main theorem is the following:
\begin{thm}\label{thm:main}
Let $n\geq 2$. The conjugacy problem for Houghton's group $H_n$ is solvable. Moreover, given two conjugate elements $g,h\in H_n$,
there is an algorithm that finds $x\in H_n$ such that $g^x=x^{-1}gx=h$.
\end{thm}

As far as we know,  Houghton's groups do not belong to a family for which the conjugacy problem is already known.
For example,  Houghton's groups are not Gromov hyperbolic or biautomatic, since by Brown's result
the $K(\pi,1)$ for $H_n$ cannot have finitely many cells in dimensions greater than $n$. Also, these groups are not conjugacy separable,
since they contain the infinite (finite support) alternating group, which is an infinite simple group.

The structure of the paper is as follows. In Section \ref{sec:comp} we describe basic computations that we can do
given a word representing an element of $H_n$, in particular, we show that the word problem is solvable and that
for each $g\in H_n$, the numbers appearing in \eqref{eq:translation} are computable. In Section \ref{sec:Fsym},
we prove that we can decide if two elements of $H_n$ are conjugate by an element of $\FSym$, and in that event,
find a conjugator in $\FSym$. Finally, in Section \ref{sec:reduce}, we show how to reduce our original problem,
to the problem about conjugators in $\FSym$.

\section{Computations in $H_n$.}\label{sec:comp}
From now on, we fix $n\in \N$, $n\geq 2$. We assume that $H= H_n$ and $X= X_n$.
We also fix $S$, a finite generating set, by setting $S=\{g_2,\dots, g_n\}$ if $n\geq 3$, and
$S= \{g_2, ((1,0)(2,0) )\}$ if $n=2$. Here the generators $g_i$ are the permutations described in \eqref{eq:generator}.

The next lemma is straightforward, however we include the proof to justify that all the computations
in the paper can be actually done.
\begin{lem}\label{lem:comp}
Let $w$ be a word over $S^{\pm 1}$ and suppose that $w$ represents $g\in H$.
Let $|w|$ denote the length of $w$ over $S$.
Then one can compute $t(g)=(t_1(g),\dots, t_n(g))$, $(i,m)g$ for all $(i,m)\in X$.
Also, one has that $|t_i(g)|<|w|$ for $i=1,\dots,n$, and moreover
$$(i,m)g=(i,m+t_i(g)),\, \forall m > |w|, i\in \{1,\ldots, n\}.$$
\end{lem}
\begin{proof}
We know that the image of $g_i$ under $t$ is $(1,0,\dots,0,-1,0,\dots,0)$, where the $-1$ is in the $i$th coordinate.
Also the image of $((1,0)(2,0))$ under $t$ is equal to the zero vector. Thus, given any word $w$ over $S^{\pm 1}$
representing $g$ one can compute $t(g)$. Clearly $|t_i(g)|<|w|$.

Similarly, given that we know the action of the generators on the elements of $X$, we can compute
$(i,m)g$ for all $(i,m)\in X$.

We prove the last claim of the lemma only in the case when $n\geq 3$ and we leave the case $n=2$ as exercise for the reader.
We argue by induction on the length of $w$.
If $|w|=0$, then $g=1$ and the claim holds. For the general case, $w=w'g_j^{\eps}$ for some $j\in \{2,\dots, n\}$, $w'$
a word over $S$ of length $|w|-1$, $g_j\in S$, $\eps\in \{ -1,+1\}$. Let $g'$ be the element of $H$ represented by $w'$.
By induction hypothesis $(i,m)g'=(i,m+t_i(g'))$, for all $i\in \{1,\dots, n\}$ and all $m>|w'|$.

Now $(i,m)g=(i,m)g'g_j^{\eps}$ and we have two cases to consider.
If $i\not\in \{1,j\}$ then $(i,m)g=(i,m)g'g_j^\eps=(i,m+t_i(g')+t_i(g_j^{\eps}))=(i,m+t_i(g'))$,
and this holds for $m>|w'|$, and in particular for $m>|w|$.
If $i=j$, then $(j,m)g=(j,m)g'g_j^\eps=(i,m+t_i(g'))g_j^{\eps}$. If $m>|w|$, then $m>|w'|+1$ and in particular,
$m+t_i(g)>0$. Thus, $(j,m)g'g_j^{\eps}=(j,m+t_j(g')-\eps)=(j,m+t_j(g))$.

The case $i=1$ is similar.
\end{proof}

Given an element of $H$ as a word $w$ over $S^{\pm 1}$ we can compute its image under $t$ and
check if it acts trivially or not on the set $\{(i,m) \in X : m\leq |w|\}$. In particular, we can
decide whether the element is trivial or not. So we have proved the following.

\begin{cor}
The word problem for $H$ is solvable.
\end{cor}

\begin{rem}
We note that the solvability of the word problem for $H$ was already known, and can be proved using a variety of techniques. It may be of interest to observe that the approach we outline here gives a quadratic time algorithm for the word problem.
\end{rem}

As a consequence of Lemma \ref{lem:comp}, it is worth noting that a number  $z$ satisfying
 equation \eqref{eq:translation} is computable, for instance, it is enough to take  $z=|w|$, where $w$ is a word representing $g$.
We note that this number is not uniquely defined from equation \eqref{eq:translation}. However, for our purposes it
will be sufficient to find any number for which equation \eqref{eq:translation} is valid.

The support of an element of $g\in H$ is the set $\{x\in X:  xg\neq x\}$, and it is denoted by $\supp(g)$.
Notice that, in general, elements in $H$ do not have finite support. However, we will show that they do have finitely many
orbits of length at least 2 (there are finitely many non-trivial finite orbits by definition, and Corollary \ref{cor:fmorbits} shows that there are
finitely many infinite orbits).

From the computability of the $t_i(g)$ and $z$ of  equation \eqref{eq:translation} it is straightforward
to determine the orbits of $g$.
Infinite orbits of $g$ will consist of pairs of the sets codified by arithmetic progressions
except for finitely many points (where the cutoff is determined by the $z$ from equation \eqref{eq:translation}).

We now show that there are finitely many infinite orbits, and hence $g$ is a product of finitely many disjoint cycles.

\begin{defn}
\label{def:xmod}
Given an element $g \in H$, for which $t_i(g)\neq 0$, we denote by $X_{i,r}(g)$ the following subset of $X$:
$$\{ (i,m) \ : \ m \equiv r \mod |t_i(g)| \}.$$
Note that these sets depend on $g$, or more precisely, the numbers $t_i(g)$.
\end{defn}
When the element $g\in H$ is clear from the context, we shall only write $X_{i,r}$ for $X_{i,r}(g)$.

Also note that for a given $g$ there are finitely many such sets,
as we can (and shall) always assume that when writing $X_{i,r}$ we have $0 \leq r < |t_i(g)|$.

One can then essentially read off the infinite orbits of $g$ using the sets $X_{i,r}$,
and therefore the infinite cycles in the cycle decomposition of $g$.
We say ``essentially'' as generally an orbit which meets some $X_{i,r}$ will contain all but finitely many points of it.

\begin{defn}
We shall say that two sets, $A,B$, are almost equal if their symmetric difference is finite.
\end{defn}

The following is then clear.

\begin{prop}
Let $g \in H$ and $(i,m) \in X$.
If the set $\{ (i,m) g^k \ : \ k \geq 0 \}$ is infinite, then it is almost equal to $X_{i,r}$ for some $i,r$ and $t_i(g)$
must be positive. Conversely, if $t_i(g)$ is positive, then $X_{i,r}$  is almost equal to
$\{ (i,m) g^k \ : \ k \geq 0 \}$ for some $m\in \N$.

Similarly, if $\{ (j,m) g^k \ : \ k \leq 0 \}$ is infinite, then it is almost equal to some $X_{j,s}$
and $t_j(g)$ is negative.
\end{prop}

\begin{cor}\label{cor:orbits}
Every infinite orbit of $g \in H$ is almost equal to $X_{i,r} \cup X_{j,s}$, for some $i,j,r,s$ with $i \neq j$.
\end{cor}

\begin{cor}\label{cor:fmorbits}
Given $g \in H_n$ there are exactly $\frac{1}{2}(\sum |t_i(g)|)$ distinct infinite orbits,
and hence the same number of infinite cycles in the cycle decomposition of $g$.
\end{cor}

Given the results above, one can determine the infinite orbits of $g$ and
hence the infinite cycles, and from there the entire cycle decomposition of $g$.

\begin{rem}
Note that in writing an element of $H$ as a product of disjoint cycles,
the cycles which appear may not themselves be elements of $H$.

Given such a decomposition as a product of disjoint cycles, it is then straightforward to decide if
two elements of $H$ are conjugate in $\Sym$; they are conjugate if and only if they have the same cycle type.
However, this will not be the same as conjugacy in $H$ in general.
\end{rem}

\section{Conjugacy by elements of $\FSym$} \label{sec:Fsym}
In this section we will prove the following
\begin{prop}\label{prop:fsym}
There is an algorithm that, given $a,b\in H$, decides if there exists $x\in \FSym\leqslant H$ such that  $b=xax^{-1}$
and produces such an element in the case it exists.
\end{prop}

Before proving the proposition, we need a lemma.
\begin{lem}\label{lem:supp}
Let $a,b\in H$. If there is $x\in \FSym\leqslant H$ such that $b=xax^{-1}$ then
\begin{equation}
\label{eq:supp}
|\supp(a)-(\supp(a)\cap \supp(b))|=|\supp(b)-(\supp(a)\cap \supp(b))|.
\end{equation}
\end{lem}
\begin{proof}
Notice that $x$ restricts to a bijection on any subset of $X$ containing $\supp(x)$, which is finite.
Also, $x$ restricts to a bijection from $\supp(a)$ to $\supp(b)$.
Thus $x$ restricts to a bijection on the set $\supp(a) \cup \supp(b) \cup \supp(x)$ which we express as
the disjoint union of four sets: $I, C_a, C_b, C_x$, where $I=\supp(a)\cap \supp(b)$,
$C_a=\supp(a)-I$, $C_b=\supp(b)-I$, $C_x=\supp(x) - (\supp(a) \cup \supp(b))$.

Notice that as $a$ and $b$ are conjugate by $x\in \FSym$, $\supp(a)$ and $\supp(b)$ are almost equal
and $\supp(x)$ is finite. Thus  $C_a$, $C_b$ and $C_x$ are finite sets.
Moreover, $x$ maps $\supp(a)$ to $\supp(b)$ and so maps
$C_a \sqcup I$ to $C_b \sqcup I$.
Since $x$ restricts to a bijection of $C_a\sqcup C_b\sqcup C_x \sqcup I$, then $x$ must map $C_a \sqcup C_x$ to $C_b \sqcup C_x$.
Thus $C_a$ and $C_b$ have the same number of elements, as required.
\end{proof}

Now we are ready to prove the proposition.

\begin{proof}[Proof of Proposition \ref{prop:fsym}]
First we notice that necessary conditions for $a$ and $b$ to be conjugate are that
$t(a)= t(b)$, and by Lemma \ref{lem:supp} that \eqref{eq:supp} holds. We notice that
both conditions are decidable. We assume that both conditions hold.

Suppose for some $i\in \{1,\dots, n\}$, that $t_i(a)=t_i(b) < 0$.
Then there exists a $z$ such that $(i,m)a=(i, m+ t_i(a))$ and
$(i,m)b=(i, m+ t_i(b))$ for all $m \geq z$ (this $z$ is computable and depends only on $a$ and $b$).
Let $x\in\FSym$ be a potential conjugator. There exists a $z'$ such that
\begin{equation}
\label{eq:zmin}
(i,m)x=(i,m) \text{ for all }m \geq z'.
\end{equation}
Let us suppose this $z'$ is minimal. We claim that $z' \leq z - t_i(a)$ (recall that $t_i(a)$ is negative).

To do this we argue by contradiction and suppose that $z' > z-t_i(a)$, and in particular, $z'>z$.
Then for any $m\geq 0$ we have that,
$$
(i,z'+m) x^{-1} a  x = (i, z'+m)b = (i, z'+m+ t_i(b)) = (i, z'+m+t_i(a))
$$
since $z'+m$ is larger than both $z$ and $z'$. On the other hand,
$$
(i,z'+m) x^{-1} a  x = (i, z'+m) a x = (i, z'+m+t_i(a)) x,
$$
and from here we deduce that, for all $m\geq 0$,
$$
(i, z'+m+t_i(a)) x = (i, z'+m+t_i(a))
$$
which contradict the minimality of $z'$ in \eqref{eq:zmin} (recall that $t_i(a)$ is negative).
A similar argument (replacing $a,b$ by their inverses) holds when $t_i(a) > 0$.
Note that any potential conjugator must send $\supp(a)$ bijectively to $\supp(b)$,
and the argument above shows that outside of a finite set whose size only depends on $a$ and $b$ (can be computed from $t$ and $z$),
a potential conjugator is the identity. In fact, a potential conjugator $x$ will be the identity
on a computable set contained in the intersection of $\supp(a)$ and $\supp(b)$ which is
almost equal to both $\supp(a)$ and $\supp(b)$. Therefore, we can decide if there exists a bijection,
$x_0$, from $\supp(a)$ to $\supp(b)$ which satisfies $x_0^{-1} a {x_0} = b$ for all elements of $\supp(b)$.
All that remains is to decide whether such an $x_0$ may be extended to a bijection on the whole of $X$.
So let us suppose that such an $x_0$ exists.

Note that {\em any} extension of $x_0$ to a bijection on $X$ will be a valid conjugator in $\Sym$,
though not necessarily in $H$. Recall that we are assuming that \eqref{eq:supp} holds, and hence, there is a bijection between the finite sets $\supp(a)-I$ and $\supp(b)-I$, where $I=\supp(a)\cap \supp(b)$. In this case,
it is clear that we can extend $x_0$ to a bijection on $\supp(a) \cup \supp(b)=\supp(a)\sqcup (\supp(b)-I)=\supp(b)\sqcup (\supp(a)-I)$, and then we can extend
it by the identity on the rest of $X$. Such a bijection is clearly an element of $\FSym$.
\end{proof}

\section{Reducing the problem to finding a conjugator in $\FSym$}\label{sec:reduce}

We now turn to the conjugacy problem for $H$. To this we consider two elements, $a,b \in H$
and try to decide whether there exists a conjugator, $x \in H$ such that $x^{-1} a x = a^x  = b$.

Note that a necessary (but not sufficient) condition for conjugacy is that $t(a)=t(b)$.
In this situation, the sets $X_{i,r}(a)$ and $X_{i,r}(b)$ of Definition \ref{def:xmod}, when defined, are the same, as these sets
only depend on $t(a)$ and $t(b)$ respectively. We shall henceforth always mean these particular sets, based on $a$ or $b$.
Throughout this section we will use $I$ to denote the set $\{i\in \{1,\dots,n\}\mid t_i(a)\neq 0\}$, that is the subset of branches
that are not almost fixed by $a$ (and by $b$).

Our first step is to reduce the problem to the case where $(X_{i,r})x$ is almost equal to $X_{i,r}$.
This is equivalent to requiring that $t_i(x) \equiv 0 \mod |t_i(a)|$ for all $i\in I$.

The following lemma is a useful observation.
\begin{lem}\label{lem:reduce}
Let $C$ be a subset of $G$. Suppose that $C=C_0\sqcup C_1\sqcup \cdots \sqcup C_k$ and that there exist
$z_1,\dots, z_k\in G$ such that $C_iz_i=C_0$.

Suppose that there is an algorithm that decides if given two elements $g,h$ of $G$ they are conjugate by an element of $C_0$ and finds this conjugator.
Then there is an algorithm that decides if given two elements $g,h$ of $G$ they are conjugate by an element of $C$ and finds this conjugator.
\end{lem}
\begin{proof}
Set $z_0=1$. Given two elements $g,h$ of $G$, the algorithm to decide conjugacy in $C$ consists in running the algorithm for conjugacy in $C_0$ for all the pairs $(g,h^{z_i})$
for $i=0,1,\dots, k$. If there are no conjugators for any of the pairs, then there is no conjugator for $g$ and $h$ in $C$ either, because if $g^x=h$ with $x$ in some $C_i$, then $g$ and $h^{z_i}$ are conjugate by $xz_i$, which is in $C_0$ by hypothesis.

And if there is a conjugator $x$ in $C_0$ for some pair $(g,h^{z_i})$, namely, $x^{-1}gx=z_i^{-1}hz_i$, then $g$ and $h$ are conjugate by $xz_i^{-1}$, which is in $C_i$, using our hypothesis, but now written as $C_i=C_0z_i^{-1}$.
\end{proof}

\begin{prop}
\label{prop:equiv}
Suppose that there is an algorithm which, given two elements $a, b \in H$ can determine whether there exists an $x \in H$ such that $a^x=b$
with $t_i(x) \equiv 0 \mod |t_i(a)|$ for all $i\in I$ and find this $x$ in case it exists.

Then, there is an algorithm such that given two elements of $H$, decides if they are conjugate in $H$, and in the event they are, the algorithm finds a conjugator.
\end{prop}
\begin{proof}
Let $R$ be the set of tuples $ \{r_i\}_{i\in I} \in \Z^{|I|}$ with $0 \leq r_i < |t_i(a)|$.
Observe that $|R|<\infty$.
Let $C=H$, and for $r\in R$, let $C_r=\{g\in H : t_i(g) \equiv r_i \mod |t_i(a)|\}$. Be denote by,  $C_0$ the set $C_{(0,0,\dots, 0)}.$
Given  a tuple $w=(w_1,\dots, w_n)\in \Z^{n}$ and $r\in R$, it is easy to decide whether or not
$w_i \equiv r_i \mod |t_i(a)|$ for all $i\in I$. Hence, we can decide whether or not a given element $g\in H$ belongs to $C_r$.

Moreover, given $r\in R$ with $\sum r_i=0$, we can easily construct $x_r\in H$
such that $t_i(x_r)=r_i$ for all $i.$

We now use Lemma \ref{lem:reduce}.
Observe that we have $C_r x_r^{-1}=C_0$ and all the hypotheses of the lemma are satisfied. This ensures the desired conclusion.
\end{proof}

We shall now attempt to find an algorithm satisfying the hypotheses of Proposition \ref{prop:equiv}.
As noted, for any potential conjugator $x$, $(X_{i,r})x$ is almost equal to $X_{i,r}$ for all $i,r$.
Since any infinite orbit of $a$ is almost equal to $X_{i,r} \cup X_{j,s}$ for some $i,r,j,s$ this means
that if $a$ and $b$ are to be conjugate via such an $x$, then every infinite orbit of $a$ must be almost equal
to some infinite orbit of $b$ and vice versa.

The key technical result is the following.

\begin{prop}
\label{prop:bound}
Let $a,b,x \in H$ and $l_1,\dots, l_n\in \Z$. Suppose that $a^x=b$ with $t_i(x) = l_i |t_i(a)|$ for all $i\in I$.
Further suppose that $X_{i,r} \cup X_{j,s}$ is almost equal to some infinite orbit of $a$ (and hence of $b$).
Then there is a (computable) constant $K=K(a,b)$, depending only on $a$ and $b$, such that  $| |l_i|- |l_j| |\leq K$.
\end{prop}

\begin{proof}

 We may algorithmically find integers $z$, $z'$ such that
$$
\begin{array}{rcl}
(k,m)a & = & (k, m+ t_k(a) ), k=i,j \ \forall m \geq z \\
(k,m)b & = & (k, m+ t_k(b) ), k=i,j \ \forall m \geq z \\
(k,m)x & = & (k, m+ t_k(x) ), k=i,j \ \forall m \geq z'.
\end{array}
$$

For convenience for what follows, since we can increase arbitrarily the number $z'$, we shall further assume that $z' \geq z+ |t_k(x)| \geq z$ for $k=i,j$,
and that $z' \equiv z$ modulo both $|t_i(a)|$ and $|t_j(a)|$.

Notice that $(X_{i,r})x$ is almost equal to $X_{i,r}$ and $(X_{j,s})x$
is almost equal to $X_{j,s}$.

Let $\cO_a$ be the infinite orbit of $a$ which is almost equal to $X_{i,r} \cup X_{j,s}$
and $\cO_b$ be the corresponding orbit for $b$. Then $X_{i,r} \cap \cO_a$ contains
$\{ (i, m) \ : \ m \equiv r \mod |t_i(a)|, m \geq z+ \epsilon \}$ for some $\epsilon$
such that $|\epsilon| \leq |t_i(a)|$ chosen such that $z + \epsilon \equiv r \mod |t_i(a)|$.
Similarly, $X_{i,r} \cap \cO_b$ also contains $\{ (i, m) \ : \ m \equiv r \mod |t_i(a)|, m \geq z+ \epsilon \}$
for the same $\epsilon$. We do the same for $j$, so that both $X_{j,s} \cap \cO_a$
and $X_{j,s} \cap \cO_b$ contain $\{ (j, m) \ : \ m \equiv s \mod |t_j(a)|, m \geq z+ \delta \}$
for some number $\delta$ such that $z + \delta \equiv s \mod |t_j(a)|$ and $|\delta| \leq |t_j(a)|$.

We can then write $\cO_a$ as the (disjoint) union of
$\{ (i, m) \ : \ m \equiv r \mod |t_i(a)|, m \geq z+ \epsilon \}$,
$\{ (j, m) \ : \ m \equiv s \mod |t_j(a)|, m \geq z+ \delta \}$
and some finite set $S$. Similarly, $O_b$ is the (disjoint)
union of $\{ (i, m) \ : \ b \equiv r \mod |t_i(a)|, m \geq z+ \epsilon \}$,
$\{ (j, m) \ : \ m \equiv s \mod |t_j(a)|, m \geq z+ \delta \}$  and some finite set $T$.

The hypotheses imply that $x$ restricts to a bijection from $\cO_a$ to $\cO_b$
and eventually acts as translation at both ends, of amplitudes $l_i$ and $l_j$.
As $z' \equiv z \mod |t_i(a)|$, we know that $x$ restricts to a bijection from
$\{ (i, m) \ : \ m \equiv r \mod |t_i(a)|, m \geq z'+ \epsilon \}$ to
$\{ (i, m) \ : \ m \equiv r \mod |t_i(a)|, m \geq z'+ \epsilon + l_i|t_i(a)| \}$.
Similarly,  $x$ also restricts to a bijection from
$\{ (j, m) \ : \ m \equiv s \mod |t_j(a)|, m \geq z'+ \delta \}$ to
$\{ (j, m) \ : \ m \equiv s \mod |t_j(a)|, m \geq z'+ \delta + l_j|t_j(g)| \}$.

Hence $x$ must also restrict to a bijection between the following finite sets:
$\{ (i, m) \ : \ m \equiv r \mod |t_i(a)|, z'+ \epsilon > m \geq z+\epsilon\}
\cup \{ (j, m) \ : \ m \equiv s \mod |t_j(a)|, z'+ \delta > m \geq z+\delta \} \cup S$
and $\{ (i, m) \ : \ m \equiv r \mod |t_i(a)|, z'+ \epsilon + l_i|t_i(a)| > m \geq z+\epsilon\}
\cup \{ (j, m) \ : \ m \equiv s \mod |t_j(a)|, z'+ \delta + l_j|t_j(a)| > b \geq z+\delta \} \cup T$.
Comparing cardinalities, we get that
$$
2(z'-z) + |S| = 2(z-z') + |T| + l_i + l_j.
$$

Hence, $|l_i| \leq |S| + |T| + |l_j| $ and $|l_j| \leq |S| + |T| + |l_i| $.
\end{proof}

In order to exploit Proposition \ref{prop:bound}, we shall define the following equivalence relation.

\begin{defn}\label{def:ends}
Given $g \in H$ we define an equivalence relation $\sim_g$ on $\{ 1, \ldots, n \}$, as the one generated by setting $i\sim_g j$ if and only if some infinite orbit of $g$
is almost equal to $X_{i,r} \cup X_{j,s}$, for some $r,s$.
\end{defn}

\begin{rem}
This equivalence relation depends on $g$.
However, if $g^x=h$ with $t_i(x) \equiv 0 \mod |t_i(g)|$ for all $i$, then $\sim_g$ and $\sim_h$ are the same  equivalence
relation.
\end{rem}

The strategy will now be to repeatedly use Proposition \ref{prop:bound}
to show that the translation lengths $t_i(x)$ for all $i\in I$ within an
equivalence class are within bounded distance of each other.
However, in order to obtain a solution to our problem, we will need to bound these lengths absolutely.
Of course, we cannot obtain this bound for any $x$, but we can prove that if the elements are conjugate, there always is a conjugator within a computable bound.
To find that, we can multiply $x$ on
the right by any element of the centraliser of $g$ without changing its properties
and we shall bound the translation lengths of $x$ after multiplying by an appropriate element.

In order to do this, we need to define certain elements of the centraliser.

\begin{lem}\label{lem:centraliser}
Let $g \in H$, with $g=\sigma_1 \ldots \sigma_p$
being the expression of $g$ as a product of disjoint cycles.
Choose $i\in\{1,2,\dots,n\}$ and let $[i]$ be its equivalence class with respect to $\sim_g$.
Let $g_i$ be the product of all cycles among the $\sigma_1,\ldots,\sigma_p$ which are infinite
and whose support is almost equal to $X_{j,r} \cup X_{k,s}$ for some $j,k\in[i]$.
Then $g_i \in H$ and is in the centraliser of $g$.
\end{lem}
\begin{proof} The fact that $g_i$ commutes with $g$ is clear.
To see that $g_i \in H$ note that if the support of $g_i$ meets $X_j$ in an infinite set,
then $j$ must be equivalent to $i$ and every infinite cycle whose support meets $X_j$
in an infinite set appears in $g_i$. Hence $g_i$ agrees with $g$ on all but finitely many points of $X_j$.
\end{proof}

Observe that by construction, the element $g_i$ produced in this lemma satisfies that $t_j(g_i)=t_j(g)$ for all $j\in[i]$.

\begin{prop}
\label{cor:zeros}
Let $a,b,x \in H$ and suppose that $a^x=b$ with $t_i(x) \equiv 0 \mod |t_i(a)|$ for all $i\in I$.
Let $i_1, \ldots, i_k$ be a set of representatives of the equivalence classes in $I$
under the equivalence relation $\sim_a$ (or $\sim_b$).
Then there exists an $x' \in H$ such that  $a^{x'}=b$ with $t_i(x') \equiv 0 \mod |t_i(a)|$
for all $i \in I$ and $t_{i_1}(x')=\ldots=t_{i_k}(x')=0$.
\end{prop}

\begin{proof} Apply Lemma \ref{lem:centraliser} to $a$ and to all the equivalence classes for $\sim_a$. For the class $[i_1]$, we construct an element $a_{i_1}$ such that $t_j(a_{i_1})=t_j(a)$ for all $j\in[i_1]$, and which is in the centraliser of $a$. Recall that the translation for $x$ satisfies $t_{i_1}(x)=l_{i_1}t_{i_1}(a)$, so we can replace $x$ by $xa_{i_1}^{-l_1}$, and this new element still conjugates $a$ into $b$, and has $t_{i_1}$ equal to zero.

The element $x'$ will then be $xa_{i_1}^{-l_1}\ldots a_{i_k}^{-l_k}$.
\end{proof}

\begin{cor}
\label{cor:bound}
Let $a,b,x \in H$ and suppose that $a^x=b$ with $t_i(x) \equiv 0 \mod |t_i(a)|$ for all $i\in I$.
Then there exists an $x' \in H$ such that $a^{x'}=b$ and such that $\sum_{i\in I} |t_i(x')|$ is
bounded by a number that depends only on $a$ and $b$ and is computable.
\end{cor}
\begin{proof} Repeated use of Proposition \ref{prop:bound} applied to the element produced by Proposition \ref{cor:zeros} will produce the desired result. Fixed an equivalence class $[i]$ of rays, there is one index $i\in [i]$ such that $t_i(x')=0$, i.e. setting $l_i=0$, $t_i(x')=l_i|t_i(a)|$. Then, for $j\in [i]$, $t_j(x')=l_j |t_j(a)|$, and $||l_i|-|l_j||=|l_j|\leq K$, where $K$ is the constant of Proposition \ref{prop:bound}.  So the total translation $\sum_{i\in I}  |t_i(x')|\leq n \cdot K\cdot M$, where $M= \max\{|t_i(a)| \,:\, i\in I\}$.
\end{proof}

We need a last lemma dealing with branches outside $I$.
\begin{lem}\label{lem:big-bound}
Let $a,b,x \in H$ and suppose that $a^x=b$ $t_i(x) \equiv 0 \mod |t_i(a)|$ for all $i\in I$.
Then there exists an $x' \in H$ such that $a^{x'}=b$ and such that $\sum_{i= 1}^n |t_i(x')|$ is
bounded by a number that depends only on $a$ and $b$ and is computable.
\end{lem}
\begin{proof}
By Corollary \ref{cor:bound}, we can assume that there exists  a computable constant $K$  depending only on $a$ and $b$ such that
$\sum_{i\in I} |t_i(x)|<K$.

Let $I^c$ denote $\{1,\dots,n\}-I=\{i_1, \ldots, i_s \}$. If $s=1$, then since $\sum_{i= 1}^n t_i(x)=0$, we get that $|t_{i_1}(x)| < K$ and hence that  $\sum_{i= 1}^n |t_i(x)|<2K$ and we are done. 

Otherwise, for each $j \neq 1$, define $y_j$ to be an element of $H_n$ such that $\supp(y_j) \cap \supp(a) = \emptyset$ and such that 
$$
\begin{array}{rcrr}
t_{i_1}(y_j) & = & t_{i_j}(x) \\
t_{i_j}(y_j) & = & -t_{i_j}(x) \\
t_i(y_j) & = & 0, & \text{ otherwise.}
\end{array}
$$

Such an element clearly exists, since $a$ fixes almost every point of the rays $i_1$ and $i_j$. 

Now let $y$ be the product of the $y_j$, and let $x'=yx$. Since each $y_j$ commutes with $a$ we have that, 
$$
a^{x'} = a^x =b.
$$

Moreover, $t_i(x')=t_i(x)$ for all $i \in I$ and $t_{i_j}(x')=0$ for $j=2, \ldots s$. Repeating the argument for when $s=1$, we get that $\sum_{i= 1}^n |t_i(x')|<2K$.
\end{proof}

We are now ready to prove the main theorem of this note.
\begin{proof}[Proof of Theorem \ref{thm:main}]
By Proposition \ref{prop:equiv}, we need to show that there is an algorithm which, given two elements $a,b\in H$, the algorithm determines whether
there exists $x\in H$ such that $a^x=b$ with $t_i(x)\equiv 0 \mod |t_i(a)|$ for all $i\in I.$
By Lemma \ref{lem:big-bound}, there is a computable number $N$, depending on $a$ and $b$ such that  if such conjugator $x$ exists,
then  $\sum |t_i(x)|<N$.

Let $S$ denote the finite family of tuples $(w_1,\dots, w_n)\in \Z^n$
satisfying $\sum |w_i|<N$, $\sum w_i=0$ and $w_i\equiv 0 \mod |t_i(a)|$ for all $i\in I.$ For each tuple $s\in S$,
we can produce an element $z_s\in H$ such that $t(z_s)=s$. We now let $C=\{x\in H : t_i(x)\equiv 0 \mod |t_i(a)|, \sum |t_i(x)|<N\}$
and for $s\in S$,  $C_s=\{x\in H : t(x)=s\}$. Set $C_0=C_{(0,0,\dots,0)}$ and observe that $C_sz_s^{-1}=C_0$.
We now use again Lemma \ref{lem:reduce}, and we conclude the desired algorithm exists, if there exists an algorithm
that can decide, for any $g,h \in H$, if there exists an $x \in H$ such that $g^x=h$ and $t_i(x) = 0$ for all $i$, and
find $x$ in case it exists. But observe that if all translations are zero, then $x\in\FSym$ and the algorithm to find it is exactly the one provided by Proposition \ref{prop:fsym} in Section \ref{sec:Fsym}.
\end{proof}

\bigskip

\noindent{\textbf{\Large{Acknowledgments}}}

Part of this work was done as at Centre de Recerca Matematica (CRM), Bellaterra in the Fall semester of 2012,
during the research program ``Automorphisms of Free Groups: Algo., Geom. and Dyn''.
The authors are grateful to CRM for its support and hospitality.

The first and second author were supported by MCI (Spain) through
project  MTM2011-25955. The first author is also supported by the swiss SNF grant: FN 200020-137696/1.

\bibliographystyle{amsplain}

\bigskip

\textsc{Yago Antol\'{i}n, University of Neuch\^{a}tel, Mathematics Department
Rue Emile-Argand 11,
CH-2000 Neuch\^{a}tel
Switzerland}

\emph{E-mail address}{:\;\;}\url{yago.anpi@gmail.com}

\emph{URL}{:\;\;}\url{https://sites.google.com/site/yagoanpi/}

\medskip

\textsc{Jos\'{e} Burillo, Departament de Matem\`{a}tica Aplicada IV,
Escola Polit\`{e}cnica Superior de Castelldefels,
Universitat Polit\`{e}cnica de Catalunya,
C/Esteve Torrades 5, 08860 Castelldefels, Barcelona, Spain}

\emph{E-mail address}{:\;\;}\url{burillo@ma4.upc.edu}

\emph{URL}{:\;\;}\url{http://www-ma4.upc.edu/~burillo/}

\medskip

\textsc{Armando Martino, School of Mathematics,
University of  Southampton, University Road,
Southampton SO17 1BJ, UK}

\emph{E-mail address}{:\;\;}\url{A.Martino@soton.ac.uk}

\emph{URL}{:\;\;}\url{http://www.personal.soton.ac.uk/am1t07/}

\end{document}

%% file: Houghton_pic.pdf_t
\begin{picture}(0,0)%
\includegraphics{Houghton_pic.pdf}%
\end{picture}%
\setlength{\unitlength}{4144sp}%
\begingroup\makeatletter\ifx\SetFigFont\undefined%
\gdef\SetFigFont#1#2#3#4#5{%
  \reset@font\fontsize{#1}{#2pt}%
  \fontfamily{#3}\fontseries{#4}\fontshape{#5}%
  \selectfont}%
\fi\endgroup%
\begin{picture}(9454,9944)(469,-8633)
\put(4276,-7711){\makebox(0,0)[lb]{\smash{{\SetFigFont{17}{20.4}{\rmdefault}{\mddefault}{\updefault}{\color[rgb]{0,0,0}$(n,4)$}%
}}}}
\put(4276,-6811){\makebox(0,0)[lb]{\smash{{\SetFigFont{17}{20.4}{\rmdefault}{\mddefault}{\updefault}{\color[rgb]{0,0,0}$(n,3)$}%
}}}}
\put(4276,-5911){\makebox(0,0)[lb]{\smash{{\SetFigFont{17}{20.4}{\rmdefault}{\mddefault}{\updefault}{\color[rgb]{0,0,0}$(n,2)$}%
}}}}
\put(4276,-5011){\makebox(0,0)[lb]{\smash{{\SetFigFont{17}{20.4}{\rmdefault}{\mddefault}{\updefault}{\color[rgb]{0,0,0}$(n,1)$}%
}}}}
\put(4276,-3211){\makebox(0,0)[lb]{\smash{{\SetFigFont{17}{20.4}{\rmdefault}{\mddefault}{\updefault}{\color[rgb]{0,0,0}$(2,0)$}%
}}}}
\put(4276,-2311){\makebox(0,0)[lb]{\smash{{\SetFigFont{17}{20.4}{\rmdefault}{\mddefault}{\updefault}{\color[rgb]{0,0,0}$(2,1)$}%
}}}}
\put(4276,-1411){\makebox(0,0)[lb]{\smash{{\SetFigFont{17}{20.4}{\rmdefault}{\mddefault}{\updefault}{\color[rgb]{0,0,0}$(2,2)$}%
}}}}
\put(4276,-511){\makebox(0,0)[lb]{\smash{{\SetFigFont{17}{20.4}{\rmdefault}{\mddefault}{\updefault}{\color[rgb]{0,0,0}$(2,3)$}%
}}}}
\put(4276,389){\makebox(0,0)[lb]{\smash{{\SetFigFont{17}{20.4}{\rmdefault}{\mddefault}{\updefault}{\color[rgb]{0,0,0}$(2,4)$}%
}}}}
\put(5176,-3436){\makebox(0,0)[lb]{\smash{{\SetFigFont{17}{20.4}{\rmdefault}{\mddefault}{\updefault}{\color[rgb]{0,0,0}$(1,0)$}%
}}}}
\put(6076,-3436){\makebox(0,0)[lb]{\smash{{\SetFigFont{17}{20.4}{\rmdefault}{\mddefault}{\updefault}{\color[rgb]{0,0,0}$(1,1)$}%
}}}}
\put(6976,-3436){\makebox(0,0)[lb]{\smash{{\SetFigFont{17}{20.4}{\rmdefault}{\mddefault}{\updefault}{\color[rgb]{0,0,0}$(1,2)$}%
}}}}
\put(7876,-3436){\makebox(0,0)[lb]{\smash{{\SetFigFont{17}{20.4}{\rmdefault}{\mddefault}{\updefault}{\color[rgb]{0,0,0}$(1,3)$}%
}}}}
\put(8776,-3436){\makebox(0,0)[lb]{\smash{{\SetFigFont{17}{20.4}{\rmdefault}{\mddefault}{\updefault}{\color[rgb]{0,0,0}$(1,4)$}%
}}}}
\put(6301,-2086){\makebox(0,0)[lb]{\smash{{\SetFigFont{17}{20.4}{\rmdefault}{\mddefault}{\updefault}{\color[rgb]{0,0,0}$g_2$}%
}}}}
\put(6301,-5461){\makebox(0,0)[lb]{\smash{{\SetFigFont{17}{20.4}{\rmdefault}{\mddefault}{\updefault}{\color[rgb]{0,0,0}$g_n$}%
}}}}
\put(4276,-4246){\makebox(0,0)[lb]{\smash{{\SetFigFont{17}{20.4}{\rmdefault}{\mddefault}{\updefault}{\color[rgb]{0,0,0}$(n,0)$}%
}}}}
\end{picture}%